\documentclass[runningheads]{llncs}
\usepackage{url}
\usepackage{multirow}
\usepackage{amssymb}
\usepackage{amsxtra}
\usepackage{makeidx}
\usepackage{amsfonts}
\usepackage{mathtools}
\usepackage{algorithm}
\usepackage{algorithmic}
\usepackage{float}
\usepackage{graphicx}
\usepackage{comment}
\usepackage[font=small,labelfont=bf]{caption}
\usepackage{subcaption}
\usepackage{breqn}
\usepackage[english]{babel}
\usepackage{ulem}

\newcommand{\al}[1]{\textcolor{black}{#1}}

\usepackage{color}
\usepackage{hyperref}

\DeclareUnicodeCharacter{202A}{\,}

\begin{document}
\title{Online Optimization Problems with Functional Constraints under Relative Lipschitz Continuity and Relative Strong Convexity Conditions\thanks{The research was supported by Russian Science Foundation (project No. 21-71- 30005), https://rscf.ru/en/project/21-71-30005/.}}
\titlerunning{Online Optimization under Relative Lipschitz Conditions}

\author{Oleg Savchuk \inst{1,2} \orcidID{0000-0003-3732-1855} \and Fedor Stonyakin  \inst{1,2} \orcidID{0000-0002-9250-4438}\and Mohammad Alkousa \inst{1,3} \orcidID{0000-0001-5470-0182} \and Rida Zabirova \inst{1} \orcidID{0000-0000-0000-0000} \and  Alexander Titov \inst{1} \orcidID{0000-0001-9672-0616} \and Alexander Gasnikov \inst{1,3,4,5} \orcidID{0000-0002-5982-8983}}
	
\authorrunning{Oleg Savchuk et al.}

\institute{Moscow Institute of Physics and Technology, 9 Institutsky lane, Dolgoprudny, 141701, Russia
	\and
	V.\,I.\,Vernadsky Crimean Federal University, 4 Academician Vernadsky Avenue, Simferopol, 295007, Republic of Crimea, Russia\\
	\and
	HSE University, Moscow, 20 Myasnitskaya street, Moscow, 101000, Russia
	\and
	Institute for Information Transmission Problems RAS, 11 Pokrovsky boulevard, 109028, Moscow, Russia
	\and
	Caucasus Mathematical Center, Adyghe State University, 208 Pervomaiskaya street, Maykop, Republic of Adygea, 385000, Russia\\
	\email{oleg.savchuk19@mail.ru, fedyor@mail.ru,
		mohammad.alkousa@phystech.edu, a.a.tytov@gmail.com, zabirova.rr@phystech.edu, gasnikov.av@phystech.edu}}

\maketitle

\begin{abstract}

A few years ago, the optimization field introduced classes of relatively smooth \cite{rel_smooth}, relatively continuous, and relatively strongly convex optimization problems \cite{LuRel,rel_strong_conv}. These concepts have expanded the class of problems to which optimal complexity estimates of gradient-type methods in high-dimensional spaces can be applied. There are known works on online optimization (regret minimization) problems for both relatively Lipschitz and relatively strongly convex problems. In this work, we consider the problem of strongly convex online optimization with convex inequality constraints. A scheme with switching \al{over} productive and non-productive steps is proposed for these problems. The convergence rate of the proposed scheme is proven for the class of relatively Lipschitz and strongly convex minimization problems. Moreover, analogously with the \cite{Hazan-Rakhlin} we study extensions of the considered Mirror Descent algorithms that eliminate the need for a priori knowledge of the lower bound on the (relative) strong convexity parameters of the observed functions. Some numerical experiments were conducted to demonstrate the effectiveness of one of the proposed algorithms with a comparison with another adaptive algorithm for convex online-optimization problems.

\keywords{Online Optimization \and Strongly Convex Programming Problem \and Relatively Lipschitz-Continuous Function \and Relatively Strongly Convex Function \and Mirror Descent \and Regularization. }
\end{abstract}


\section*{Introduction}\label{sec1_introduction}
The development of numerical methods for solving non-smooth online optimization problems presents a great interest nowadays due to the appearance of many applied problems with the corresponding statement \cite{Bubeck-Bianchi,Hazan-Rakhlin,Hazan-Kale,Hazan-2016,Lugosi-Bianchi}. Online optimization plays a key role in solving machine learning, finance, networks, and other problems. As some examples of such problems, we can mention multi-armed bandits, job-shop scheduling and ski rental problems, search games, etc. One of the most popular methods of solving online optimization problems is the Mirror Descent method \cite{Orabona_online_2015}. Let us note, that Mirror Descent can be also applied for solving online optimization problems in a stochastic setting \cite{article:mohammad_kim_2019,Gasnikov_stoc_online_2017},  which allows using an arbitrary, not necessarily $1$--strongly convex, distance-generating function (see \eqref{Bregman}).

Remind, that the online optimization problem represents the problem of minimizing the sum (or the arithmetic mean) of $T$ functionals
$f_t:Q\longrightarrow \mathbb{R}$ ($t = \overline{1, T}$) given on some closed convex set $ Q \subset \mathbb{R}^n $
\begin{equation}\label{Problem_Statement}
\min\limits_{x \in Q} \frac{1}{T}\sum\limits_{t=1}^T f_t(x),  \;\; s.t. 
 \;\; g(x)\leq 0.
\end{equation}
The key feature of the problem statement consists in the possibility of calculating the (sub)gradient $\nabla f_t(x)$ of each functional $f_t$ only once.

Recently, in \cite{Titov_online} there were proposed some modifications of the Mirror Descent method for solving online optimization problems in the case, if all the functions $f_t(x)$ and functional constraint $g(x)$ satisfy Lipschitz condition, i.e. there exists such a constant $M > 0$, that
\begin{equation}
|g(x)-g(y)|\leq M\|x-y\|,
\end{equation}
\begin{equation}
|f_t(x)-f_t(y)|\leq M\|x-y\|, \quad \forall t = \overline{{1, T}}.
\end{equation}
In the case of non-negativity of regret
\begin{equation}
Regret_T:=\sum_{t=1}^T f_t(x_t) - \min_{x\in Q} \sum_{t=1}^T f_t(x),
\end{equation}
these methods are optimal for the considered class of problems accordingly to \cite{Hazan-Kale}, the number of non-productive steps during their work is $O(T)$. In the case of negative regret, the number of non-productive steps for the proposed methods is $O(T^2)$.

Later, in \cite{Titov_online_model} the smoothness class for the applicability of such approaches has been extended by reducing the requirement of Lipschitz continuity of functions to the recently proposed concept of relative Lipschitz continuity \cite{Lu,Nesterov_Relative_Smoothness}.

\begin{definition}\label{1}
Let us call a convex function $f:Q \longrightarrow \mathbb{R}$ $M$-relatively Lipschitz-continuous for some $M>0$, if the following inequality holds
\begin{equation}\label{X}
 \langle\nabla f(x), y - x\rangle + M\sqrt{2V(y,x)} \geq 0, \quad \forall x, y \in Q.
\end{equation}
\end{definition}
This concept has been widely used in many applied problems and has also enabled the proposal of subgradient methods for both non-differentiable and non-Lipschitz Support Vector Machine (SVM) and for problems of Intersection of $n$ Ellipsoids while maintaining optimal convergence rate estimates for the class of simply Lipschitz-continuous functions. It is worth noting that the proposed method also allowed the use of an imprecisely defined function (more exactly, a function that admits a representation in a model form), nevertheless, the method was also optimal.

Let $h: Q \longrightarrow \mathbb{R}$ be a distance-generating function (or prox-function) that is continuously differentiable and convex.
For all $x,y\in Q$ we consider the corresponding Bregman divergence
\begin{equation}\label{Bregman}
    V(y, x) = h(y) - h(x) -\langle \nabla h(x),y-x\rangle.
\end{equation}

In this paper, we improve existing estimates of the convergence rate by considering a class of strongly convex functions and generalize the obtained problem statement to the case of problems with functional constraints.
\begin{definition}\label{2}
A function $f$ over a convex set $Q$ is called $\mu$-strongly convex with respect to a convex function $h$ if
\begin{equation*}
f(x) \ge f(y)+\langle\nabla f(y),x-y\rangle+\mu V(x,y),\quad \forall x,y\in Q,
\end{equation*}
\end{definition}

\al{More precisely, we present} a novel theorem that provides a tighter bound on~regret, in terms of the number of productive steps taken by the algorithm. Specifically, the theorem proves that if the algorithm completes exactly $T$ productive steps and has a non-negative regret, then the number of non-productive steps 
\al{satisfies} $T_J\leq CT$, where $C$ is a constant. This result significantly improves 
\al{existing} convergence rate estimates for the Mirror Descent method with functional constraints. In addition, \al{we obtain} the complexity of the bound in terms of $T$ and \al{some} other problem parameters. This corollary allows us to determine the number of productive steps needed to achieve the desired \al{accuracy} of regret \al{in practice}. 


\al{We also consider} \al{some} modifications \al{of} the Mirror Descent method for solving non-smooth online optimization problems \cite{Hazan-Rakhlin}. Specifically, the paper introduces two algorithms for solving strongly convex minimization problems with and without regularization. The first algorithm, called General-Norm Online Gradient Descent: Relatively Strongly Convex and Relatively Lipschitz-Continuous Case, is based on a convex function $h$ and updates the solution iteratively using predictions and observations of the objective function $f_t$. The second algorithm, called Adaptive General-Norm Online Gradient Descent with Regularization, extends the first algorithm by introducing an adaptive regularization term that depends on a function $d$ that is both relative\al{ly} Lipschitz continuous and relative\al{ly} strongly convex.

For each algorithm, \al{we provide the} theoretical justification of bounds on the regret. These theorems \al{guarantee} upper bounds on the regret for each algorithm and can be used to analyze the performance of the algorithms. Overall, the paper presents a comprehensive framework for solving non-smooth online optimization problems with functional constraints, and the results have practical implications for a broad range of applications.

The paper consists of an introduction and 4 main sections. In Sect. \ref{section_1} we  consider the basic statement of the constrained online optimization problem and propose a modification of the Mirror Descent method for minimizing the arithmetic mean of relatively strongly convex and relatively Lipschitz-continuous functionals, supposing that functional constraint satisfies the same conditions. We also provide a theoretical justification for the convergence rate of the proposed method. Sect. \ref{section_regulirization} is devoted to some modifications of the algorithms, proposed in \cite{Hazan-Rakhlin} for the corresponding class of problems with regularization. In Sect. \ref{section_online_with_constr} we combine the above-mentioned ideas and propose algorithms with switching \al{over} productive and non-productive steps both with and without iterative regularization during the work of algorithms. In Sect. \ref{section_experemints} we present some numerical experiments which demonstrate the effectiveness of one of the proposed algorithms and a comparison with another adaptive algorithm for the considered optimization problems.

To sum it up, the contributions of the paper can be stated as follows:
\begin{itemize}
\item We proposed an optimal method for solving a constrained online optimization problem with relatively strongly convex and relatively Lipschitz-continu-ous objective functionals and functional constraints. For the case of non-negative regret, the number of non-productive steps is bounded by $O(T)$.

\item We proposed two algorithms for solving strongly convex minimization problems with and without regularization based on iteratively updating steps by using some auxiliary functions. \al{Similar to} \cite{Hazan-Rakhlin}, we \al{present} extensions of Mirror Descent that \al{exclude} the need for a priori knowledge of the lower bound on the (relatively) strong convexity parameters of the observed functions. 
	
\item We provided the results of numerical experiments demonstrating the advantages of using the proposed methods.

\end{itemize}

\section{Mirror Descent for Relatively Strongly Convex and Relatively Lipschitz-Continuous Online-optimization Problems with Inequality Constraints}\label{section_1}
In this section, we present a scheme with switching \al{over} productive and non-productive steps for relatively strongly convex and relatively Lipschitz-continuous online optimization problems with inequality constraints. We consider the following \al{strongly convex constrained optimization} problem 
\begin{equation}\label{pr_1}
\min_{x\in Q} \sum_{t=1}^T f_t(x),\quad  g(x) \leq 0,
\end{equation}
where $f_t: Q \longrightarrow \mathbb{R}$ and $g: Q \longrightarrow \mathbb{R}$. Let $x^*$ be a solution of \eqref{pr_1}, i.e.
$$
x^*=\arg\min\limits_{x\in Q} \sum_{t=1}^T f_t(x),\quad  g(x^*) \leq 0.
$$
Let \al{us} denote the set of productive steps $x_t$ for which $g(x_t) \leq \varepsilon$ by $I$, and the set of non-productive steps by $J$. Let $T=|I|, T_J=|J|.$ \al{Let us} consider a subgradient method with switching \al{over} productive and non-productive steps.
\begin{algorithm}
\caption{ Constrained Online Optimization: Mirror Descent for Relatively Lipschitz-Continuous and Relatively Strongly Convex Problems. }
\label{alg:myaaa}
\begin{algorithmic}[1]
\REQUIRE $\varepsilon>0, \mu>0,  T,  x_1 \in Q$.
\STATE $i:= 1, t:=1$;
\REPEAT
\IF{$g(x_t) \leq \varepsilon$}
\STATE $\eta_t = \frac{1}{\mu t}$;
\STATE $x_{t+1}:= Pr_Q\{x_t - \eta_t\nabla f_t(x_t)\}$; \quad ''productive step''
\STATE $i:= i+1$;
\STATE $t:= t+1$;
\ELSE
\STATE $\eta_t = \frac{1}{\mu t}$;
\STATE $x_{t+1}:= Pr_Q\{x_t - \eta_t\nabla g(x_t)\}$;  \quad ''non-productive step''
\STATE $t:= t+1$;
\ENDIF
\UNTIL{$i=T+1$}.
\end{algorithmic}
\end{algorithm}

\begin{theorem}
Suppose that, for each $t$, $f_t$ is an $M_f$-relatively Lipschitz continuous and $\mu$-strongly convex function with respect to the prox-function $h$. Let $g(x)$ be $M_g$-relatively Lipschitz continuous and $\mu$-strongly convex function with respect to $h$. Suppose that Algorithm \ref{alg:myaaa} for
$$
\varepsilon = \dfrac{M^2}{\mu} \dfrac{1 + \ln T}{T}
$$
where $M = \max\{M_f,M_g\},$ works exactly $T$ productive steps and $Regret_T \geq 0$. Then there exists a constant $C \in (2; 3)$ such that the number of non-productive steps \al{satisfies} $T_J \leq CT$, \al{moreover,} the following inequality holds:
$$
Regret_T:=\sum_{t=1}^T f_t(x_t) - \min_{x\in Q} \sum_{t=1}^T f_t(x) \leq  \frac{M^2}{\mu}\Bigg(1+\ln\Big((C+1)T\Big)\Bigg) = O(T\varepsilon),
$$
where $g(x_t) \leq \varepsilon$ for any $t = \overline{1, T}$.
\end{theorem}

\begin{proof}
Let us check the auxiliary inequality
\begin{equation}\label{eneq1}
   \sum_{t=1}^T f_t(x_t) - \min_{x\in Q} \sum_{t=1}^T f_t(x) \leq \frac{M^2}{\mu}(1+\ln(T+T_J)) - \varepsilon T_J.
\end{equation}
\begin{enumerate}
\item Taking into account the $M_f$-relative Lipschitz-continuity of the function $f_t$ for  \al{each} productive step we have
\begin{equation*}
    \begin{aligned}
        \eta_t\Big(f_t(x_t) -f_t(x^*)\Big) & \leq \eta_t\Big(\langle\nabla f_t, x_t -x^*\rangle - \mu V(x^*,x_t)\Big)
       \\&
       \leq \eta_t^2M_f^2+V(x^*,x_t)-V(x^*,x_{t+1})-\eta_t\mu V(x^*,x_t).
    \end{aligned}
\end{equation*}
Hence, after dividing both sides of the above inequality by $\eta_t$ we get
\begin{equation}\label{100024}
    \begin{aligned}
        f_t(x_t) -f_t(x^*) & \leq \eta_tM_f^2+\frac{1}{\eta_t}\Big(V(x^*,x_t)-V(x^*,x_{t+1})\Big)-\mu V(x^*,x_t)
        \\&
        = \frac{M_f^2}{\mu t} + \mu tV(x^*,x_t)-\mu V(x^*,x_t)-\mu tV(x^*,x_{t+1})
        \\&
        =\frac{M_f^2}{\mu t} + \mu(t-1) V(x^*,x_t)-\mu tV(x^*,x_{t+1}).
    \end{aligned}
\end{equation}

\item Similarly, taking into account the $M_g$-relative Lipschitz-continuity of $g$ for \al{each} \al{non-}productive step we have $g(x_t) > \varepsilon$ and
\begin{equation*}
    \begin{aligned}
      \eta_t\varepsilon & < \eta_t(g(x_t) - g(x^*))\leq \eta_t \left(\langle\nabla g, x_t -x^*\rangle - \mu V(x^*,x_t) \right)
      \\&
      \leq \eta_t^2M_g^2+V(x^*,x_t)-V(x^*,x_{t+1})-\eta_t\mu V(x^*,x_t).
    \end{aligned}
\end{equation*}
Dividing both sides of the last inequality by $\eta_t$, we get:
\begin{equation}\label{12545}
    \begin{aligned}
        \varepsilon & < g(x_t) - g(x^*) 
        \\& \leq \eta_tM_g^2+\frac{1}{\eta_t}\Big(V(x^*,x_t)-V(x^*,x_{t+1})\Big)-\mu V(x^*,x_t)
        \\&
        = \frac{M_g^2}{\mu t} + \mu tV(x^*,x_t)-\mu V(x^*,x_t)-\mu tV(x^*,x_{t+1})
        \\&
        =\frac{M_g^2}{\mu t} + \mu(t-1) V(x^*,x_t)-\mu tV(x^*,x_{t+1}).
    \end{aligned}
\end{equation}


\item Summing up inequalities \eqref{100024}, \eqref{12545}  \al{over} productive and  \al{non-}productive steps,  for $M = \max\{M_f,M_g\}$, we get
\begin{equation*}
    \begin{aligned}
        & \quad  \sum_{t \in I} \Big(f_t(x_t) - f_t(x^*)\Big) + \sum_{t \in J} \Big(g(x_t) - g(x^*)\Big)
        \\&
        \leq  \sum_{t=1}^{ T+T_J} \left( \frac{M^2}{\mu t} + \mu(t-1) V(x^*,x_t)-\mu tV(x^*,x_{t+1})\right)
        \\&
        \leq \frac{M^2}{\mu}\ln(T + T_J) - \mu(T+T_J)V(x^*,x_{T+T_J})
        \\&
        \leq \frac{M^2}{\mu}\ln(T + T_J).
    \end{aligned}
\end{equation*}

Using the fact, that for non-productive steps
\[ g(x_t) - g(x^*) \ge g(x_t) > \varepsilon, \]
we get an estimate for the sum of the objective functionals:
\begin{equation*}
    \begin{aligned}
        \sum_{t \in I} \Big(f_t(x_t) - f_t(x^*) \Big) &  \leq \frac{M^2}{\mu}\ln(T + T_J) - \sum_{t \in J} \Big(g(x_t) - g(x^*)\Big)
        \\&
        \leq \frac{M^2}{\mu}\ln(T + T_J) - \sum_{t \in J} \varepsilon
        = \frac{M^2}{\mu}\ln(T + T_J) - \varepsilon T_J.
    \end{aligned}
\end{equation*}

\item According to the assumption of non-negativity of the regret, we find
\begin{equation*}
    \begin{aligned}
        0 \leq Regret_T & = \sum_{t=1}^T \Big(f_t(x_t) - f_t(x^*)\Big) = \sum_{t=1}^T f_t(x_t) - \min_{x\in Q}\sum_{t=1}^T f_t(x)
        \\&
        \leq \frac{M^2}{\mu}\Big(1+\ln(T+T_J)\Big) - \varepsilon T_J.
    \end{aligned}
\end{equation*}
\al{Hence}
$\varepsilon T_J \leq \dfrac{M^2}{\mu}\Big(1+\ln(T+T_J)\Big)$ and $\varepsilon = \dfrac{M^2}{\mu} \dfrac{1 + \ln T}{T}$. Therefore, we have
$$
    \dfrac{1 + \ln T}{T} T_J \leq 1+\ln(T+T_J),
$$
$$
    \dfrac{T_J}{T}  \leq \dfrac{1+\ln(T+T_J)}{1 + \ln T}.
$$
Moreover, taking into account
$$
    \ln(T+T_J) = \ln\Bigg(T\left(1+\frac{T_J}{T}\right)\Bigg) = \ln T + \ln\left(1+\frac{T_J}{T}\right),
$$
we get
$$
    \dfrac{T_J}{T}  \leq \dfrac{1+\ln T + \ln(1+\frac{T_J}{T})}{1 + \ln T} \leq 1 + \ln\left(1+\frac{T_J}{T}\right).
$$

Since the linear function grows faster than the logarithmic \al{one}, it is obvious, that \al{for} a sufficiently large $T_J$, the above inequality does not hold, \al{thus}  $\dfrac{T_J}{T}$ is bounded. \al{Therefore,} we proved that $T_J = O(T)$, i.e. \al{there exists} $C>0$ \al{such that} $T_J \leq C T$ or $\frac{T_J}{T} \leq C$:
$$
    \dfrac{T_J}{T} \leq 1 + \ln\left(1+\frac{T_J}{T}\right).
$$

Equality in the latter inequality is achieved when
$$
    \dfrac{T_J}{T} \approx 2,146. 
$$


5. Further, we note that by the definition of $\varepsilon$, we have
$$
\varepsilon = \dfrac{M^2}{\mu} \dfrac{1 + \ln T}{T} = \dfrac{M^2}{\mu T} + \dfrac{M^2}{\mu} \dfrac{\ln T}{T}.
$$
Since $T$ is the number of productive steps and $T_J \leq CT$ is the number of non-productive steps, the total number of steps is $T + T_J \leq (C+1)T$. Therefore
$$
\varepsilon = \dfrac{M^2}{\mu (T + T_J)} + \dfrac{M^2}{\mu} \dfrac{\ln (T + T_J)}{T + T_J} \leq \dfrac{M^2}{\mu (C+1)T} + \dfrac{M^2}{\mu} \dfrac{\ln (C+1)T}{(C+1)T}.
$$
This allows us to bound the regret as follows:
$$
Regret_T := \sum_{t=1}^T f_t(x_t) - \min_{x\in Q} \sum_{t=1}^T f_t(x) \leq
\dfrac{M^2}{\mu} \Big( 1 + \ln (C+1)T \Big).
$$
This shows that the bound on the regret, given by the last inequality holds, which \al{finishes} the proof.
\end{enumerate}
\end{proof}
\begin{remark}
\al{Let us} show that our algorithm will necessarily make \al{at least one} productive steps. Indeed, \al{suppose,}  \al{that the number of productive steps equals zero}, then
\[
    \varepsilon T_J \leq \sum_{t=1}^{T_J} \Big(g(x_t) - g(x^*)\Big) \leq \dfrac{M^2}{\mu}\Big(1+\ln T_J\Big). 
\]
It is obviously, that \al{for} a sufficiently large $T_J$, the above inequality does not hold. Thus, for a sufficiently large number of non-productive steps, there will be at least one productive \al{step}.

\al{Let us} find out how many non-productive steps need to be taken to achieve inequality:
\[\varepsilon T_J = \dfrac{T_J M^2}{\mu} \dfrac{1 + \ln T}{T}  > \dfrac{M^2}{\mu}(1 + \ln T_J), \]
\[  \dfrac{1 + \ln T}{T} >   \dfrac{1 + \ln T_J}{T_J}. \]

Then $T_J\leq C  T$, \al{where $C$ is a constant,} which proves that the number of non-productive steps is \al{bounded}  until at least one productive \al{step} is \al{made}.
\end{remark}

\section{Online Mirror Descent with Regularization}\label{section_regulirization}

In this section, we propose some modifications of the algorithms \al{proposed} in \cite{Hazan-Rakhlin} for relatively strongly convex and relatively Lipschitz online optimization problems and \al{provide theoretical estimates} of the quality of the solution.

We consider the following  strongly convex minimization problem
\begin{equation}\label{pr_2}
\min_{x\in Q} \sum_{t=1}^T f_t(x),
\end{equation}
where $f_t : Q \longrightarrow \mathbb{R}$. 
Define $\mu_{1:t} : = \sum\limits_{s=1}^{t} \mu_s $, where $\mu_s$ is the parameter of relative strong convexity of the function $f_s$. Let $\mu_{1:0}=0.$
\begin{algorithm}[!ht]
\caption{General-Norm Online Gradient Descent: Relatively Strongly Convex and Relatively Lipschitz-Continuous Case.}\label{alg_1}
\begin{algorithmic}[1]
   \STATE Input: convex function $h$.
	\STATE Initialize $x_1$ arbitrarily.
	\FOR{$t=1, \ldots, T$}
    \STATE Predict $x_t$, observe $f_t$.
    \STATE Compute $\eta_{t+1}$ and let $y_{t+1}$ be such that $\nabla h(y_{t+1})=\nabla h(x_{t})- \eta_{t+1}\nabla f_t(x_t)$.
	\STATE Let $x_{t+1}=\arg\min\limits_{x\in Q}V(x,y_{t+1})$ be the projection of $y_{t+1}$ onto $Q$.
	\ENDFOR
\end{algorithmic}
\end{algorithm}

\begin{theorem}\label{t_1}
Suppose that, for each $t$, $f_t$ is an $M_t$-relatively Lipschitz-continuous and $\mu_t$-strongly convex function with respect to prox-function $h$. Applying the Algorithm \ref{alg_1} with $\eta_{t+1}=\frac{1}{\mu_{1:t}}$, we have
$$
Regret_T\le \sum\limits_{t=1}^{T} \frac{M_t^2}{\mu_{1:t}}.
$$
\end{theorem}
\begin{proof}
The proof is given in Appendix A.
\end{proof}

Let’s now consider an analogue of Algorithm \ref{alg_1} for relatively strongly convex and relatively Lipschitz-continuous problems with iterative regularization. Define $\lambda_{1:t}:=\sum\limits_{s=1}^{t}\lambda_s$. The proposed algorithm is listed as Algorithm \ref{alg_2}, below. 

\begin{algorithm}[!ht]
\caption{Adaptive General-Norm Online Gradient Descent with Regularization.}\label{alg_2}
\begin{algorithmic}[1]
   \STATE Input: convex function $h$.
    \STATE Initialize $x_1$ arbitrarily.
    \FOR{$t = 1, \ldots,  T$}
    \STATE Predict $x_t$, observe $f_t$.
    \STATE Compute $\lambda_t=\frac{1}{2}\left(\sqrt{(\mu_{1:t}+\lambda_{1:t-1})^2+8M_t^2/(A^2+2M_d^2)}-(\mu_{1:t}+\lambda_{1:t-1})\right).$
    \STATE Compute $\eta_{t+1}$ and let $y_{t+1}$ be such that $$
        \nabla h(y_{t+1})=\nabla h(x_{t})- \eta_{t+1}\left(\nabla f_t(x_t)+\lambda_t \nabla d(x_t)\right).
    $$
	\STATE Let $x_{t+1}=\arg\min\limits_{x\in Q}V(x,y_{t+1})$ be the projection of $y_{t+1}$ onto $Q$.
    \ENDFOR
\end{algorithmic}
\end{algorithm}

For Algorithm \ref{alg_2}, we have the following result. 

\begin{theorem}\label{t_2}
Suppose that, for each $t$, $f_t$ is $M_t$-relatively Lipschitz-continuous and $\mu_t$-relatively strongly convex function with respect to the prox-function $h$. Let $d : Q \longrightarrow \mathbb{R}$ be $M_d$-relatively Lipschitz-continuous and $1$-strongly convex function with respect to $h$. Suppose that $d(x)\ge 0, \; \forall x\in Q$ and $A^2=\sup\limits_{x\in Q} d(x)$. Applying Algorithm \ref{alg_2} with $\eta_{t+1}=\frac{1}{\mu_{1:t}+\lambda_{1:t}}$, the following inequalities hold
$$
    Regret_T\le\lambda_{1:T}A^2+\sum\limits_{t=1}^{T}\frac{(M_t+\lambda_tM_d)^2}{\mu_{1:t}+\lambda_{1:t}},
$$
and
$$
    Regret_T\le2\inf\limits_{\lambda_1^*,\dots,\lambda_T^*}\left((A^2+2M_d^2)\lambda_{1:T}^*+\sum\limits_{t=1}^{T}\frac{(M_t+\lambda_t^*M_d)^2}{\mu_{1:t}+\lambda_{1:t}^*}\right).
$$
\end{theorem}
\begin{proof}
The proof is given in Appendix B.
\end{proof}

\section{The Case of Online Optimization Problems with Functional Constraints}\label{section_online_with_constr}

In this section, we consider a scheme with switching \al{over}  productive and non-productive steps both with and without iterative regularization for a relatively strongly convex and relatively Lipschitz-continuous \al{constrained} online optimization problem.

Remind that we consider the following problem of strongly convex conditional minimization
$$
\min_{x\in Q}\sum_{t=1}^T f_t(x),\quad  g(x) \leq 0.
$$
and
$$x^*=\arg\min\limits_{x\in Q} \sum_{t=1}^T f_t(x),\quad  g(x^*) \leq 0,$$
where $f_t : Q \longrightarrow \mathbb{R}$  and $g : Q \longrightarrow \mathbb{R}.$ 
Remind that the set of productive steps is $I$, the set of non-productive steps is $J$ and $T=|I|, T_J=|J|.$ Similarly to Section \ref{section_regulirization}, we define $\mu_{1:t}:=\sum\limits_{s=1}^{t}\mu_s$, where $\mu_s$ is the parameter of relative strong convexity of the function $f_s$ and let $\mu_{1:0}=0.$ If $t$ is the number of non-productive step, then $\mu_t=\mu_g,$ where $\mu_g$ is the parameter of relative strong convexity of the function $g.$
\begin{algorithm}
\caption{Mirror Descent for Constrained Optimization Problems with Relatively Lipschitz-Continuous and Relatively Strongly Convex Functions.}
\label{alg:myaab}
\begin{algorithmic}[1]
\REQUIRE $\varepsilon>0, T, x_1 \in Q$.
\STATE $i:= 1, t:=1$;
\REPEAT
\IF{$g(x_t) \leq \varepsilon$}
\STATE $\eta_t = \frac{1}{\mu_{1:t}}$;
\STATE $x_{t+1}:= Pr_Q\{x_t - \eta_t\nabla f_t(x_t)\}$; \quad ''productive step''
\STATE $i:= i+1$;
\STATE $t:= t+1$;
\ELSE
\STATE $\eta_t = \frac{1}{\mu_{1:t}}$;
\STATE $x_{t+1}:= Pr_Q\{x_t - \eta_t\nabla g(x_t)\}$;  \quad  ''non-productive step''
\STATE $t:= t+1$;
\ENDIF
\UNTIL{$i=T+1$}.
\STATE Guaranteed accuracy:
    $$
        \delta:=\frac{1}{T}\left(\sum_{t=1}^{  T+T_J}\frac{M^2}{\mu_{1:t}} - \varepsilon T_J\right).
    $$
\end{algorithmic}
\end{algorithm}

\begin{theorem}\label{Thm_4}
Suppose that, for each $t$, $f_t$ is an $M_t$-relatively Lipschitz-continuous and $\mu_t$-strongly convex function with respect to the convex function $h$. Let $g(x)$ be $M_g$-relatively Lipschitz-continuous and $\mu_g$-strongly convex function with respect to $h$. If Algorithm \ref{alg:myaab} works exactly $T$ productive steps and $Regret_T \geq 0$, then the following inequality holds:
$$
Regret_T\le \sum_{t=1}^{ T+T_J}\frac{M^2}{\mu_{1:t}} - \varepsilon T_J\al{,}
$$
where $M = \max\{M_t,M_g\}$ and $g(x_t) \leq \varepsilon$ for any $t = \overline{1, T}$.
\end{theorem}

\begin{proof}
\begin{enumerate}
\item Taking into account that $f_t$ is $M_t$-relative Lipschitz continuous, then for every productive step, we have
\begin{equation*}
    \begin{aligned}
    \eta_t\Big(f_t(x_t) -f_t(x^*)\Big) & \leq \eta_t\Big(\langle\nabla f_t, x_t -x^*\rangle - \mu_t V(x^*,x_t)\Big) \\&
    \leq \eta_t^2M_t^2+V(x^*,x_t)-V(x^*,x_{t+1})-\eta_t\mu_t V(x^*,x_t).
    \end{aligned}
\end{equation*}
Hence, after dividing both sides of the above inequality by $\eta_t$\al{,} we get
\begin{equation}\label{eq_1000}
    \begin{aligned}
    f_t(x_t) -f_t(x^*) & \leq \eta_tM_t^2+\frac{1}{\eta_t}\Big(V(x^*,x_t)-V(x^*,x_{t+1})\Big)-\mu_t V(x^*,x_t) \\&
    = \frac{M_t^2}{\mu_{1:t}} + \mu_{1:t}V(x^*,x_t)-\mu_t V(x^*,x_t)-\mu_{1:t}V(x^*,x_{t+1}) \\&
    =\frac{M_t^2}{\mu_{1:t}} + \mu_{1:t-1}V(x^*,x_t)-\mu_{1:t}V(x^*,x_{t+1}).
    \end{aligned}
\end{equation}

\item Similarly, taking into account that  $g$ is $M_g$-relative Lipschitz continuous, then for every non-productive step\al{,} we have $g(x_t) > \varepsilon$, and
\begin{equation*}
    \begin{aligned}
    \eta_t \varepsilon & < \eta_t\Big(g(x_t) - g(x^*)\Big)\leq \eta_t \Big(\langle\nabla g, x_t -x^*\rangle - \mu_t V(x^*,x_t) \Big)  \\&
    \leq \eta_t^2M_g^2+V(x^*,x_t)-V(x^*,x_{t+1})-\eta_t\mu_tV(x^*,x_t).
    \end{aligned}
\end{equation*}
Dividing both sides of the last inequality by $\eta_t$, we get:
\begin{equation}\label{eq_5000}
    \begin{aligned}
    \varepsilon & < g(x_t) - g(x^*)
    \\&
    \leq \eta_tM_g^2+\frac{1}{\eta_t}\Big(V(x^*,x_t)-V(x^*,x_{t+1})\Big)-\mu_tV(x^*,x_t)
    \\&
    = \frac{M_g^2}{\mu_{1:t}} + \mu_{1:t}V(x^*,x_t)-\mu_tV(x^*,x_t)-\mu_{1:t}V(x^*,x_{t+1})
    \\&
    =\frac{M_g^2}{\mu_{1:t}} + \mu_{1:t-1}V(x^*,x_t)-\mu_{1:t}V(x^*,x_{t+1}).
    \end{aligned}
\end{equation}



\item Summing up inequalities \eqref{eq_1000}, \eqref{eq_5000} \al{over} productive and non-productive steps\al{,}  \al{for} $M = \max\{M_t,M_g\}$,  \al{we get}
\begin{equation*}
    \begin{aligned}
        & \sum_{t \in I} \Big(f_t(x_t) - f_t(x^*) \Big) + \sum_{t \in J} \Big(g(x_t) - g(x^*)\Big)
        \\&
        \leq  \sum_{t=1}^{ T+T_J} \left( \frac{M^2}{\mu_{1:t}} + \mu_{1:t-1}V(x^*,x_t)-\mu_{1:t}V(x^*,x_{t+1})\right)
        \\&
         \leq \sum_{t=1}^{ T+T_J}\frac{M^2}{\mu_{1:t}} - \mu_{1:T+T_J}V(x^*,x_{T+T_J}) \leq \sum_{t=1}^{ T+T_J}\frac{M^2}{\mu_{1:t}}.
    \end{aligned}
\end{equation*}

Using the fact, that for non-productive steps
\[ g(x_t) - g(x^*) \ge g(x_t) > \varepsilon, \]
we get an estimate for the sum of the objective functionals:
\begin{equation*}
    \begin{aligned}
        \sum_{t \in I} \Big(f_t(x_t) - f_t(x^*) \Big) &
        \leq \sum_{t=1}^{ T+T_J}\frac{M^2}{\mu_{1:t}} - \sum_{t \in J} \Big(g(x_t) - g(x^*)\Big)
        \\&
        \leq \sum_{t=1}^{  T+T_J}\frac{M^2}{\mu_{1:t}} - \sum_{t \in J} \varepsilon= \sum_{t=1}^{ T+T_J}\frac{M^2}{\mu_{1:t}} - \varepsilon T_J.
    \end{aligned}
\end{equation*}

\item Thus, we get
\begin{equation*}
    \begin{aligned}
        0 \leq Regret_T=\sum_{t=1}^T \Big(f_t(x_t) - f_t(x^*)\Big) & = \sum_{t=1}^T f_t(x_t) - \min_{x\in Q}\sum_{t=1}^T f_t(x)\\&
        \leq \sum_{t=1}^{  T+T_J}\frac{M^2}{\mu_{1:t}} - \varepsilon T_J.
    \end{aligned}
\end{equation*}

\end{enumerate}
\end{proof}

\begin{corollary}\label{cor_1}
Assume that all conditions of Theorem \ref{Thm_4} hold and suppose \\$\mu_t\ge \mu>0$ for all $1\leq t\leq T+T_J$. If $$\varepsilon = \dfrac{M^2}{\mu}\dfrac{1 + \ln T}{T},$$ then the bound on the regret of Algorithm \ref{alg:myaab} is $O(\ln T)$.
\end{corollary}

\begin{proof}
$$
0\leq Regret_T\leq \sum_{t=1}^{ T+T_J}\frac{M^2}{\mu_{1:t}} - \varepsilon T_J\leq\sum_{t=1}^{ T+T_J}\frac{M^2}{\mu t} - \varepsilon T_J\leq \dfrac{M^2}{\mu}\Big(\ln (T+T_J)+1\Big)-\varepsilon T_J,
$$
hence
$\varepsilon T_J \leq\dfrac{M^2}{\mu}\Big(1+\ln(T+T_J)\Big)$. Let $\varepsilon = \dfrac{M^2}{\mu}\dfrac{1 + \ln T}{T}.$
Then we have
$$\dfrac{1 + \ln T}{T} T_J \leq 1+\ln(T+T_J),$$
and
$$\dfrac{T_J}{T}  \leq \dfrac{1+\ln(T+T_J)}{1 + \ln T}=\dfrac{1+\ln T + \ln(1+\frac{T_J}{T})}{1 + \ln T} \leq 1 + \ln(1+\frac{T_J}{T}).$$
Since the linear function grows faster than the logarithmic one, it is obviously, that with a sufficiently large $T_J$, the above inequality does not hold, and then $\dfrac{T_J}{T}$ is bounded. Thus we proved that \al{there exists such a constant $C>0$, that} $T_J \leq CT$.
So, we have $$ Regret_T\leq \dfrac{M^2}{\mu}\Bigg(1+\ln\Big((C+1)T\Big)\Bigg)=O(\ln T)=O(T\varepsilon).$$
\end{proof}
Let’s consider an analogue of Algorithm \ref{alg:myaab} for relatively strongly convex and relatively Lipschitz-continuous problems with iterative regularization. Similarly to Section \ref{section_regulirization}, we define $\lambda_{1:t}:=\sum\limits_{s=1}^{t}\lambda_s$.
\begin{algorithm}[!ht]
\caption{ Constrained Online Optimization: Mirror Descent for Relatively Strongly Convex and Relatively Lipschitz-Continuous Problems with Regularization.}
\label{alg:myaac}
\begin{algorithmic}[1]
\REQUIRE $\varepsilon>0, x_1 \in Q$.
\STATE $i:= 1, t:=1$;
\REPEAT
\IF{$g(x_t) \leq \varepsilon$}
\STATE $
\lambda_t=\frac{1}{2}\left(\sqrt{(\mu_{1:t}+\lambda_{1:t-1})^2+8M^2/(A^2+2M_d^2)}-(\mu_{1:t}+\lambda_{1:t-1})\right)
$;
\STATE $\eta_t = \frac{1}{\mu_{1:t}+\lambda_{1:t}}$;
\STATE $x_{t+1}:= Pr_Q\{x_t - \eta_t(\nabla f_t(x_t)+\lambda_t\nabla d(x_t))\}$; \quad  ''productive step''
\STATE $i:= i+1$;
\STATE $t:= t+1$;
\ELSE
\STATE $
\lambda_t=\frac{1}{2}\left(\sqrt{(\mu_{1:t}+\lambda_{1:t-1})^2+8M^2/(A^2+2M_d^2)}-(\mu_{1:t}+\lambda_{1:t-1})\right)
$;
\STATE $\eta_t = \frac{1}{\mu_{1:t}+\lambda_{1:t}}$;
\STATE $x_{t+1}:= Pr_Q\{x_t - \eta_t(\nabla g(x_t)+\lambda_t\nabla d(x_t))\}$; \quad  ''non-productive step''
\STATE $t:= t+1$;
\ENDIF
\UNTIL{$i=T+1$}.
\STATE Guaranteed accuracy:
    $$
        \delta:=\frac{1}{T}\left(\lambda_{1:T+T_J}A^2+\sum\limits_{t=1}^{T+T_J}\frac{(M+\lambda_tM_d)^2}{\mu_{1:t}+\lambda_{1:t}}-\varepsilon T_J\right).
    $$
\end{algorithmic}
\end{algorithm}

\begin{theorem}\label{Thm5}
Suppose that, for each $t$, $f_t$ is an $M_t$-relatively Lipschitz-continuous and $\mu_t$-relatively strongly convex function with respect to the prox-function $h$. Let $g(x)$ be $M_g$-relatively Lipschitz-continuous and $\mu_g$-relatively strongly convex function with respect to $h$. Let $d : Q \longrightarrow \mathbb{R}$  be $M_d$-relatively Lipschitz-continuous and $1$-relatively strongly convex function with respect to $h$. Suppose also that $d(x)\ge 0, \; \forall x\in Q$. If Algorithm \ref{alg:myaac} works exactly $T$ productive steps and $Regret_T \geq 0$, then the following inequalities hold:
$$
Regret_T\le\lambda_{1:T+T_J}A^2+\sum\limits_{t=1}^{T+T_J}\frac{(M+\lambda_tM_d)^2}{\mu_{1:t}+\lambda_{1:t}}-\varepsilon T_J,
$$
and
$$
Regret_T\le2\inf\limits_{\lambda_1^*,\dots,\lambda_{T+T_J}^*}\left((A^2+2M_d^2)\lambda_{1:T+T_J}^*+\sum\limits_{t=1}^{T+T_J}\frac{(M+\lambda_t^*M_d)^2}{\mu_{1:t}+\lambda_{1:t}^*}\right)-\varepsilon T_J.
$$
where $A^2=\sup\limits_{x\in Q} d(x)$, $M = \max\{M_t,M_g\}$ and $g(x_t) \leq \varepsilon$ for any $t = \overline{1, T}$.
\end{theorem}
\begin{proof}
The proof is given in Appendix C.
\end{proof}

We can formulate the following statement for concrete values of $\mu_t$.
Partially, we can achieve intermediate rates for regret between
$T$ and $log T$.

\begin{corollary}\label{cor_4}
Assume that all conditions of Theorem \ref{Thm5} hold and $\mu_t=t^{-\alpha}$ for all $1\leq t\leq T+T_J.$
\begin{enumerate}
\item If $\alpha=0, \lambda_t=0 \; \forall 1\leq t\leq T+T_J$,  and $\varepsilon = M^2\dfrac{1 + \ln T}{T},$ then the bound on the regret of Algorithm \ref{alg:myaac} is $O(\ln T)$.
\item If $\alpha>1/2, \lambda_1=\sqrt{T+T_J}, \lambda_t=0$ for $1<t\leq T+T_J,$ and $$\varepsilon = \dfrac{A^2+2(M_d^2+M^2)}{\sqrt{T}},$$ then the bound on the regret of Algorithm \ref{alg:myaac} is $O(\sqrt{T})$.
\item If $0<\alpha\leq 1/2,\lambda_1=(T+T_J)^{\alpha}, \lambda_t=0\quad\forall 1\leq t\leq T+T_J$ and
$$
\varepsilon = \left(A^2 + 2 M_d^2 + \frac{4 M^2}{\alpha}\right)T^{\alpha - 1},
$$ 
then the bound on the regret of Algorithm \ref{alg:myaac} is $O(T^{\alpha})$.
\end{enumerate}
\end{corollary}
\begin{proof}
The proof is given in Appendix D.
\end{proof}

\section{Numerical Experiments}\label{section_experemints}
In this section, to demonstrate the performance of the proposed Algorithm \ref{alg:myaab}, we conduct some numerical experiments for the \al{considered} problem  \eqref{Problem_Statement} and \al{make}  a comparison with an adaptive Algorithm 2, proposed in \cite{Titov_online}.
All experiments were implemented in Python 3.4, on a computer fitted with Intel(R) Core(TM) i7-8550U CPU @ 1.80GHz, 1992 Mhz, 4 Core(s), 8 Logical Processor(s). RAM of the computer is 8 GB.

Let us consider the following function
\begin{equation}\label{objective_experiments}
    f(x) = \frac{1}{T} \sum_{i = 1}^{T} \left( \left|\langle a_i, x \rangle - b_i\right| + \frac{\mu_i}{2} \|x\|_2^2 \right),
\end{equation}
where $a_i \in \mathbb{R}^n, b_i \in \mathbb{R}, \mu_i>0 $.  \al{F}unctional constraints \al{are defined as follows} 
\begin{equation}\label{constraints_g}
    g(x) = \max_{1 \leq i \leq m} \left\{ \langle  \alpha_i, x\rangle - \beta_i + \frac{ \widehat{\mu_i} }{2} \|x\|_2^2 \right\},
\end{equation}
where $\alpha_i \in \mathbb{R}^n, \beta_i \in \mathbb{R}, \widehat{\mu_i}  >0 $.

\al{F}unction $f$ is  \al{the} \al{arithmetic mean}  of the  functions $f_i(x) = \left|\langle a_i, x \rangle - b_i\right| + \frac{\mu_i}{2}\|x\|_2^2, \; i=\overline{1, T}$. Each of these functions is $M_i$-Lipschitz-continuous and $\mu_i$-strongly convex. Also, function $g$ is $M_g$-Lipschitz-continuous and $\mu_g$-strongly convex. \al{C}oefficients $a_i, \alpha_i \in \mathbb{R}^n$ and constants $b_i, \beta_i \in \mathbb{R}$ in \eqref{objective_experiments} and \eqref{constraints_g}  are randomly generated from the uniform distribution over $[0,1)$. Also, the strong convexity parameters $\mu_i$ and $\widehat{\mu}_i$ are randomly chosen in the interval $(0,1)$.

We choose a standard Euclidean proximal setup as a prox-function, starting point $x_0 = \left( \frac{1}{\sqrt{n}}, \ldots, \frac{1}{\sqrt{n}}\right) \in \mathbb{R}^n$ and $Q$ is the unit ball in $\mathbb{R}^n$.

We run Algorithm \ref{alg:myaab} and adaptive Algorithm 2 from \cite{Titov_online}  with $n = 1000$ and $m= 10$ and different values of $T$ with $\varepsilon = 1/\sqrt{T}$. The results of the work of these algorithms are represented in Fig. \ref{fig_results}, below. These results demonstrate the number of non-productive steps, the running time is given in seconds,  the guaranteed accuracy $\delta $ of the approximated solution (sequence $\{x_t\}_{t \in I}$ on productive steps), and the values $\frac{1}{T} \sum_{i= 1}^{T}f_i(x_i)$, where $x_i$ is productive, as a function of $T$. The dotted curve represents the results of the proposed Algorithm \ref{alg:myaab}, whereas the dashed curve represents the results of the adaptive Algorithm 2 in \cite{Titov_online}.

From the conducted experiments, we can see that the adaptive Algorithm 2 in \cite{Titov_online}, works faster than Algorithm \ref{alg:myaab}, with a smaller amount of non-productive steps. But when increasing the number of functional\al{s} $f_i$ in \eqref{objective_experiments}, the guaranteed accuracy $\delta$ and values of the objective function at productive steps, \al{produced} by Algorithm \ref{alg:myaab} is better.

Note that from Fig. \ref{fig_results}, we can see that \al{increasing of} $T$ (the number of functionals $f_i$) \al{leads to an increasing of}  $\delta $ \al{(the accuracy of the solution).} 
\al{In other words, increasing the number of functionals $f_i$ in the objective function \eqref{objective_experiments}, which in fact is increasing information about the objective function or actually enlarging data about the problem, leads to increasing the accuracy of the solution.}

\begin{figure}[h]
	\centering
	{\includegraphics[width=6cm]{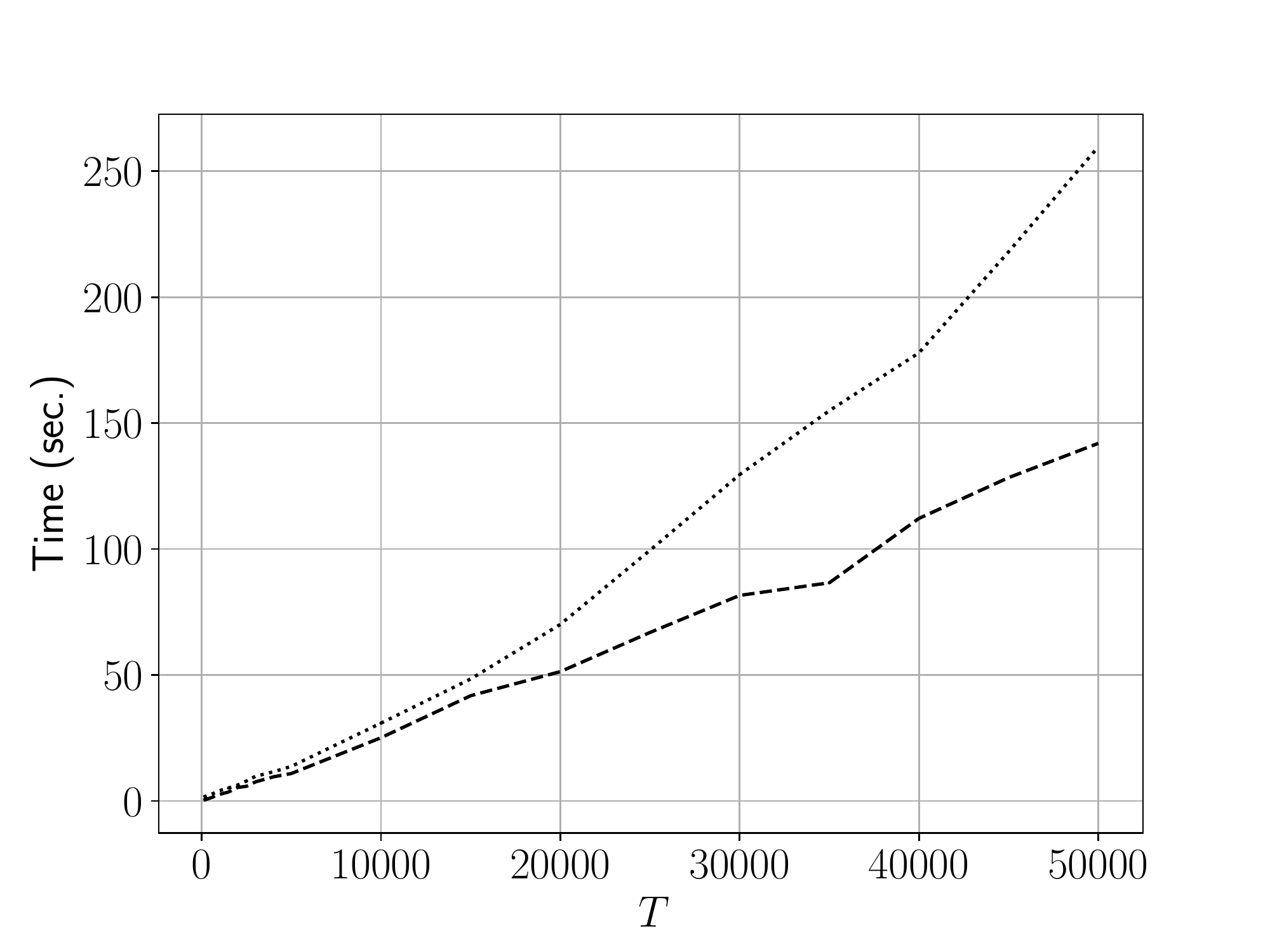} }
	{\includegraphics[width=6cm]{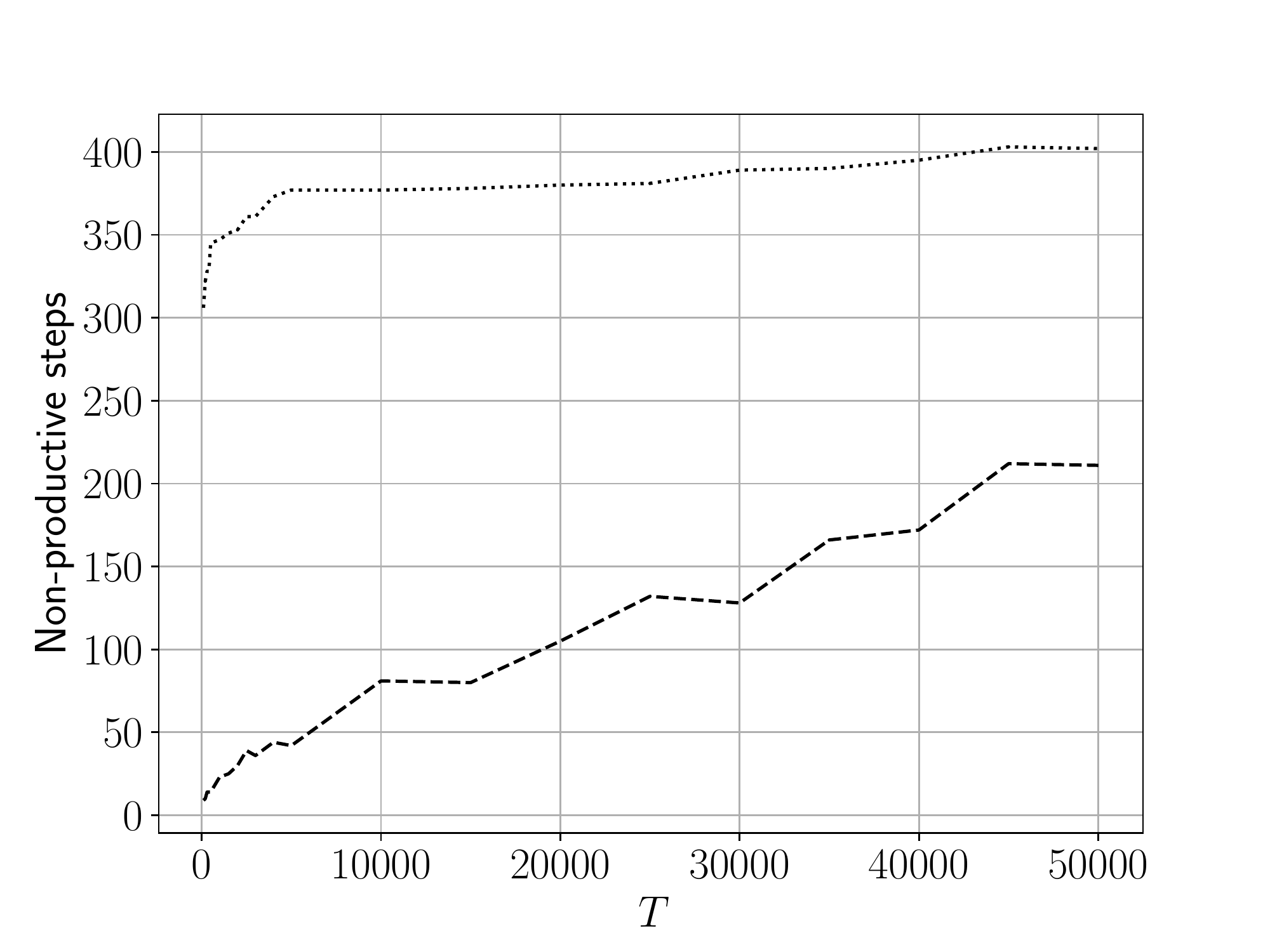} }
	{\includegraphics[width=5.9cm]{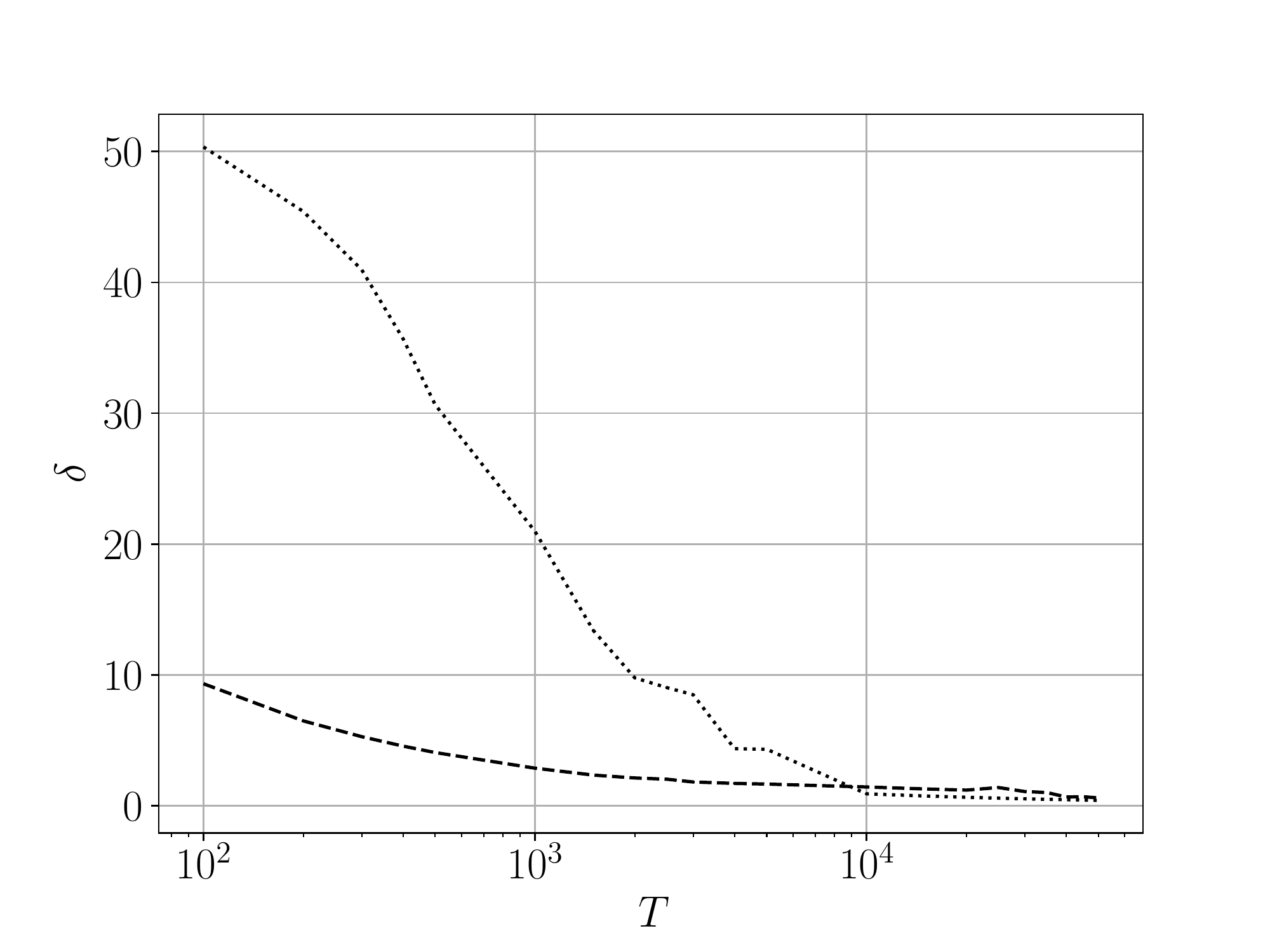} }
         {\includegraphics[width=5.9cm]{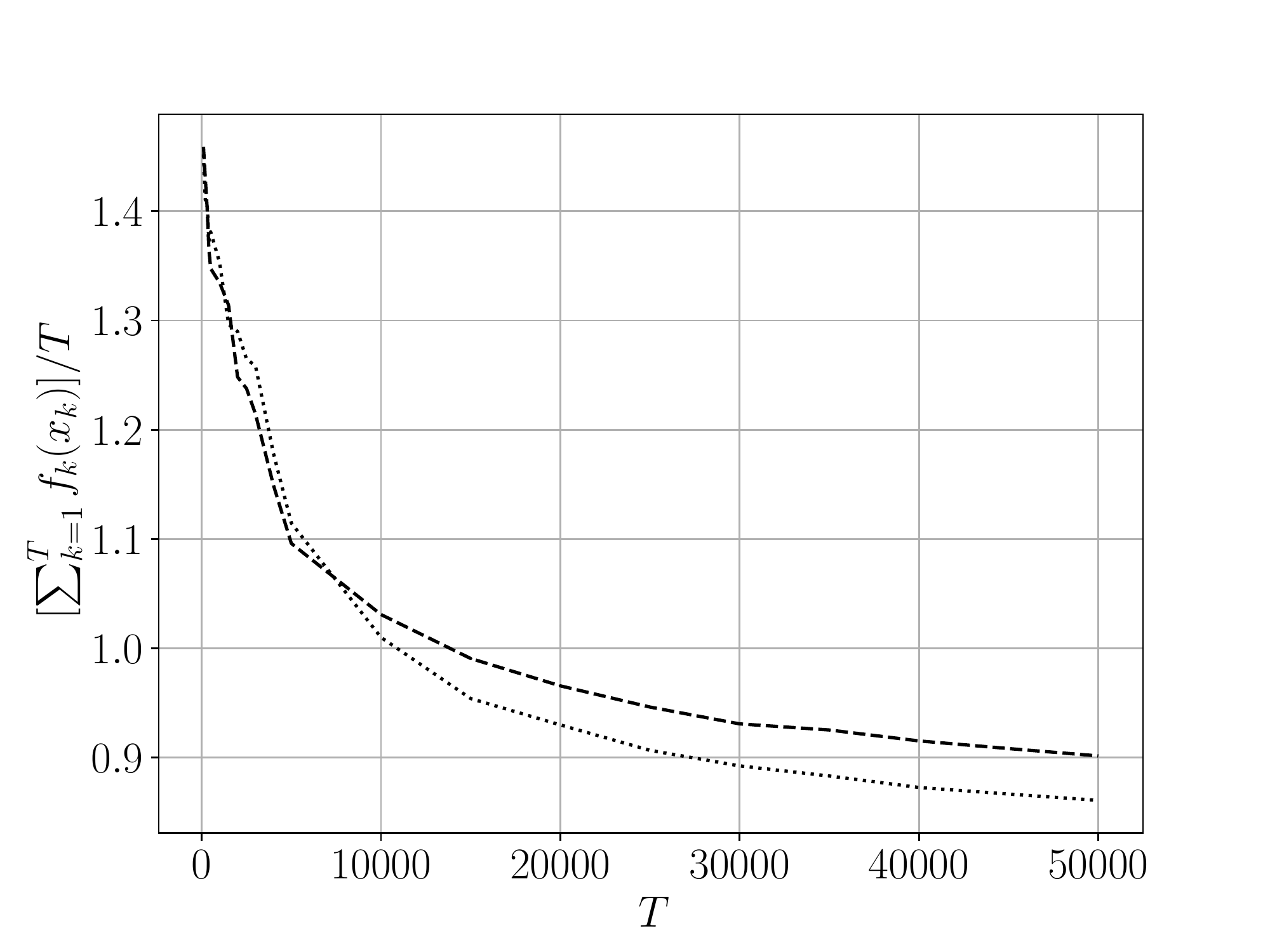} }
	\caption{The results of Algorithm \ref{alg:myaab} (dots) and adaptive Algorithm 2 in \cite{Titov_online} (dashed) for the objective function \eqref{objective_experiments} with constraints \eqref{constraints_g}.}
	\label{fig_results}
\end{figure}

\section*{Conclusions}
In this paper, we considered relatively strongly convex and relatively Lipschitz-continuous \al{constrained} online optimization problems. \al{W}e  proposed some methods with switching \al{over} productive and non-productive steps \al{and provided corresponding}  estimates of the quality of the solution. \al{We also} presented analogues of the methods proposed earlier in \cite{Hazan-Rakhlin}, for solving relatively strongly convex and relatively Lipschitz-continuous online optimization problems with and without regularization. Furthermore, for the problems with functional constraints, we have proposed a scheme with switching \al{over}  productive and non-productive steps with adaptive regularization. We also proved that if the algorithm \al{runs} exactly $T$ productive steps and has a non-negative regret, then the number of non-productive steps  \al{satisfies} $T_J\leq C T$, where $C$ is a constant. In particular, for the proposed methods, we obtained some bounds on the algorithm’s regret in terms of the number of productive steps  \al{made} by the algorithm  under specific assumptions about the parameters of relative strong convexity and \al{some} other parameters of the problem.

The key idea of the considered methods is that at each step of the algorithm for each selected $f_t$, we determine the corresponding parameter of the relative strong convexity $\mu_t$. Thus, it is possible to take into account the parameter of relative strong convexity  \al{of} each of the functions $f_t$. This is highly significant because the functions are selected during the method's working process, and it would be a mistake to assume that some strong convexity can be set initially. It is important to note, that if we consider \al{the following functional constraint} 
$$
g(x)=\max\limits_{1 \leq i \leq m}\{g_i(x)\},
$$
where each $g_i$ is $\mu_i$-relatively strongly convex function, then in the process of working of the algorithm at this particular non-productive step $t$, it makes sense to consider the first of the constraints $g_i(x)$ for which the condition $g_i(x_t)\leq \varepsilon$ is violated and \al{the} corresponding  parameter $\mu_i$, i.e. $\mu_t=\mu_i$. We do not initially know which constraint will be violated in the process of working of the method, and it is logical to take into account its relative strong convexity parameter instead of the global relative strong convexity \al{one}, which may turn out to be much larger. We have analyzed the results of the given numerical experiments and compared the effectiveness of one of the proposed algorithms with Algorithm 2 proposed in \cite{Titov_online}.

\newpage

\section*{Appendix A. The proof of Theorem \ref{t_1}.}
\begin{proof}
    
By the assumption on the functions $f_t$, for $x^*=\arg\min\limits_{x\in Q} \sum_{t=1}^T f_t(x)$ we have
$$
f_t(x_t)-f_t(x^*)\le \langle\nabla f_t(x_t),x_t - x^*\rangle-\mu_tV(x^*,x_t).
$$
By a well-known property of Bregman divergences, it holds that for any vectors $x,y,z$,
$$
\langle x-y,\nabla h(z)-\nabla h(y)\rangle=V(x,y)-V(x,z)+V(y,z).
$$
Combining both observations,
\begin{equation*}
    \begin{aligned}
        f_t(x_t)-f_t(x^*) & \le \langle\nabla f_t(x_t),x_t - x^*\rangle-\mu_tV(x^*,x_t)
        \\&
        =\frac{1}{\eta_{t+1}}\langle\nabla h(y_{t+1})-\nabla h(x_t),x^*-x_t\rangle-\mu_tV(x^*,x_t)
        \\&
        =\frac{1}{\eta_{t+1}}\left[V(x^*,x_t)-V(x^*,y_{t+1})+V(x_t,y_{t+1})\right]-\mu_tV(x^*,x_t)
        \\&
        \le\frac{1}{\eta_{t+1}}\left[V(x^*,x_t)-V(x^*,x_{t+1})+V(x_t,y_{t+1})\right]-\mu_tV(x^*,x_t),
    \end{aligned}
\end{equation*}
where the last inequality follows from the Pythagorean Theorem for Bregman divergences, as $x_{t+1}$ is the projection w.r.t the Bregman divergence of $y_{t+1}$ and $x^*\in Q$ is in the convex set.

Summing over all iterations and recalling that $\eta_{t+1}=\frac{1}{\mu_{1:t}}$,
\begin{equation}\label{X1}
\begin{aligned}
Regret_T & \le \sum\limits_{t=2}^{T} V(x^*,x_t)\left(\frac{1}{\eta_{t+1}}-\frac{1}{\eta_{t}} - \mu_t\right)+V(x^*,x_1)\left(\frac{1}{\eta_2}-\mu_1\right)
\\&
\;\;\;\; +\sum\limits_{t=1}^{T}\frac{1}{\eta_{t+1}}V(x_t,y_{t+1})
=\sum\limits_{t=1}^{T}\frac{1}{\eta_{t+1}}V(x_t,y_{t+1}).
\end{aligned}
\end{equation}
We procced to bound $V(x_t,y_{t+1})$. By the definition of Bregman divergence, and the $M_t$-relative Lipschitz-continuity,
\begin{equation*}
\begin{aligned}
    V(x_t,y_{t+1})+V(y_{t+1},x_t) & =\langle\nabla h(x_{t})-\nabla h(y_{t+1}),x_t-y_{t+1}\rangle
    \\&
    =\eta_{t+1}\langle\nabla f_t(x_t),x_t-y_{t+1}\rangle
    \\&
    \le \eta_{t+1} M_t \sqrt{2 V(y_{t+1},x_t)}
    \\&
    =\sqrt{2M_t^2\eta_{t+1}^2V(y_{t+1},x_t)}
    \\&
    \le M_t^2\eta_{t+1}^2+V(y_{t+1},x_t).
\end{aligned}
\end{equation*}
Thus, we have
$$
V(x_t,y_{t+1})\le M_t^2\eta_{t+1}^2.
$$
Plugging back into \eqref{X1} we get
$$
Regret_T\le\sum\limits_{t=1}^{T}\frac{1}{\eta_{t+1}}V(x_t,y_{t+1})\le \sum\limits_{t=1}^{T}\eta_{t+1}\cdot M_t^2=\sum\limits_{t=1}^{T}\frac{M_t^2}{\mu_{1:t}}.
$$
\end{proof}

\section*{Appendix B. The proof of Theorem \ref{t_2}.}

At the first, let us mention the following auxiliary lemma, which was proposed in \cite{Hazan-Rakhlin}.
\begin{lemma}\label{l_1}
Define
$$
H_T (\{\lambda_t\}) = H_T (\lambda_1, \dots,\lambda_T) = \lambda_{1:T} + \sum\limits_{t=1}^{T} \frac{C_t}{\mu_{1:t} + \lambda_{1:t}},
$$
where $C_t\ge 0$ does not depend on $\lambda_t.$ If $\lambda_t$ satisfies $\lambda_t=\frac{C_t}{\mu_{1:t}+\lambda_{1:t}}$ for $t=1,\dots,T,$ then
$$
H_T(\{\lambda_t\})\le 2\inf\limits_{\{\lambda_t^*\}\ge 0}H_T(\{\lambda_t^*\}).
$$
\end{lemma}

Now, let us prove Theorem \ref{t_2}.

\begin{proof}
By assumption on the functions $f_t$ and $d$, for $x^*=\arg\min\limits_{x\in Q} \sum_{t=1}^T f_t(x)$ we have
$$
    f_t(x_t)-f_t(x^*)\le \langle\nabla f_t(x_t),x_t - x^*\rangle-\mu_tV(x^*,x_t),
$$
and
$$
    d(x_t)-d(x^*)\le \langle\nabla d(x_t),x_t - x^*\rangle-V(x^*,x_t).
$$
Summing these two inequalities, we have
\begin{equation*}
    \begin{aligned}
        (f_t(x_t)+\lambda_t d(x_t))-(f_t(x^*)+\lambda_td(x^*)) & \le \langle\nabla f_t(x_t)+\lambda_t\nabla d(x_t),x_t - x^*\rangle
        \\&
         \;\;\;\; -(\mu_t+\lambda_t)V(x^*,x_t).
    \end{aligned}
\end{equation*}

By a well-known property of Bregman divergences, it holds that for any vectors $x,y,z$,
$$
    \langle x-y,\nabla h(z)-\nabla h(y)\rangle=V(x,y)-V(x,z)+V(y,z).
$$
Combining both observations,
\begin{equation*}
    \begin{aligned}
        & \quad  (f_t(x_t)+\lambda_t d(x_t))-(f_t(x^*)+\lambda_td(x^*))
        \\& \le \langle\nabla f_t(x_t)+\lambda_t\nabla d(x_t),x_t - x^* \rangle
        - (\mu_t + \lambda_t) V(x^*,x_t)
        \\&
        =\frac{1}{\eta_{t+1}}\langle\nabla h(y_{t+1})-\nabla h(x_t),x^*-x_t\rangle-(\mu_t+\lambda_t)V(x^*,x_t)
        \\&
        =\frac{1}{\eta_{t+1}}\left[V(x^*,x_t)-V(x^*,y_{t+1})+V(x_t,y_{t+1})\right]-(\mu_t+\lambda_t)V(x^*,x_t)
        \\&
        \le\frac{1}{\eta_{t+1}}\left[V(x^*,x_t)-V(x^*,x_{t+1})+V(x_t,y_{t+1})\right]-(\mu_t+\lambda_t)V(x^*,x_t),
    \end{aligned}
\end{equation*}

where the last inequality follows from the Pythagorean theorem for Bregman divergences, as $x_{t+1}$ is the projection w.r.t the Bregman divergence of $y_{t+1}$ and $x^*\in Q$ is in the convex set.

Summing over all iterations and recalling that $\eta_{t+1}=\frac{1}{\mu_{1:t}+\lambda_{1:t}}$,
\begin{equation}\label{X2}
\begin{aligned}
 & \sum\limits_{t=1}^{T}(f_t(x_t)+\lambda_t d(x_t))-\sum\limits_{t=1}^{T} (f_t(x^*)+\lambda_t d(x^*))
 \\& \le
 \sum\limits_{t=2}^{T}V(x^*,x_t)\left(\frac{1}{\eta_{t+1}}-\frac{1}{\eta_{t}}-\mu_t-\lambda_t\right) + V(x^*,x_1)\left(\frac{1}{\eta_2}-\mu_1-\lambda_1\right)
\\&
\;\;\;\; +\sum\limits_{t=1}^{T}\frac{1}{\eta_{t+1}}V(x_t,y_{t+1})
= \sum\limits_{t=1}^{T}\frac{1}{\eta_{t+1}}V(x_t,y_{t+1}).
\end{aligned}
\end{equation}
We proceed to bound $V(x_t,y_{t+1})$. By the definition of Bregman divergence, and the relative Lipschitz-continuity,
\begin{equation*}
    \begin{aligned}
        V(x_t,y_{t+1})+V(y_{t+1},x_t) & = \langle\nabla h(x_{t})-\nabla h(y_{t+1}),x_t-y_{t+1}\rangle
        \\&
        =\eta_{t+1}\langle\nabla f_t(x_t) +\lambda_t \nabla d(x_t),x_t-y_{t+1}\rangle
        \\&
        \le \eta_{t+1}M_t\sqrt{2V(y_{t+1},x_t)}
        +\lambda_t\eta_{t+1}M_d\sqrt{2V(y_{t+1},x_t)}
        \\&
        =(M_t+\lambda_tM_d)\sqrt{2\eta_{t+1}^2V(y_{t+1},x_t)} \\&
        = \sqrt{2 (M_t+\lambda_tM_d)^2 \eta_{t+1}^2V(y_{t+1},x_t)}
        \\&
        \le (M_t + \lambda_t M_d)^2 \eta_{t+1}^2 + V(y_{t+1}, x_t).
    \end{aligned}
\end{equation*}
Thus, we have
$$
V(x_t,y_{t+1})\le (M_t+\lambda_tM_d)^2\eta_{t+1}^2.
$$
Plugging back into \eqref{X2} we get
\begin{equation*}
\begin{aligned}
    & \sum\limits_{t=1}^{T}(f_t(x_t)+\lambda_t d(x_t))-\sum\limits_{t=1}^{T} (f_t(x^*)+\lambda_t d(x^*))\le \sum\limits_{t=1}^{T}\frac{1}{\eta_{t+1}}V(x_t,y_{t+1})\le
    \\&
    \le\sum\limits_{t=1}^{T}\eta_{t+1} (M_t+\lambda_tM_d)^2=\sum\limits_{t=1}^{T}\frac{(M_t+\lambda_tM_d)^2}{\mu_{1:t}+\lambda_{1:t}}.
\end{aligned}
\end{equation*}
Thus, we have
$$
\sum\limits_{t=1}^{T}(f_t(x_t)+\lambda_t d(x_t))\leq\min\limits_x\left(\sum\limits_{t=1}^{T} (f_t(x)+\lambda_t d(x))\right)+\sum\limits_{t=1}^{T}\frac{(M_t+\lambda_tM_d)^2}{\mu_{1:t}+\lambda_{1:t}}.
$$
Dropping the $d(x_t)$ terms and bounding $d(x^*)\le A^2$, we have
$$
\sum\limits_{t=1}^{T}f_t(x_t)\le\sum\limits_{t=1}^{T} f_t(x^*)+\lambda_{1:T}A^2+\sum\limits_{t=1}^{T}\frac{(M_t+\lambda_tM_d)^2}{\mu_{1:t}+\lambda_{1:t}},
$$
hence
\begin{equation}\label{tttt}
    Regret_T \le \lambda_{1:T}A^2 + \sum\limits_{t=1}^{T} \frac{(M_t+ \lambda_t M_d)^2}{\mu_{1:t} + \lambda_{1:t}}.
\end{equation}

The following inequality allows us to remove the dependence on $\lambda_t$ from the numerator of the second sum in \eqref{tttt}. We have
\begin{equation}\label{X3}
\begin{aligned}
\lambda_{1:T}A^2+\sum\limits_{t=1}^{T}\frac{(M_t+\lambda_tM_d)^2}{\mu_{1:t}+\lambda_{1:t}} & \le \lambda_{1:T}A^2+\sum\limits_{t=1}^{T}\left(\frac{2M_t^2}{\mu_{1:t}+\lambda_{1:t}}+\frac{2\lambda_t^2M_d^2}{\mu_{1:t}+\lambda_{1:t-1}+\lambda_t}\right)
\\&
\le (A^2+2M_d^2)\lambda_{1:T}+2\sum\limits_{t=1}^{T}\frac{M_t^2}{\mu_{1:t}+\lambda_{1:t}}.
\end{aligned}
\end{equation}

By \eqref{X3} and Lemma \ref{l_1}, we have
\begin{equation*}
    \begin{aligned}
        Regret_T & \le (A^2+2M_d^2) \lambda_{1:T} + 2 \sum\limits_{t=1}^{T}\frac{M_t^2}{\mu_{1:t}+\lambda_{1:t}}
        \\&
        \le \inf\limits_{\lambda_1^*, \dots,\lambda_T^*} \left(2(A^2+2M_d^2)\lambda_{1:T}^* + 4 \sum\limits_{t=1}^{T} \frac{M_t^2}{\mu_{1:t} + \lambda_{1:t}^*}\right)
        \\&
        \le 2 \inf\limits_{\lambda_1^*, \dots,\lambda_T^*} \left((A^2 + 2 M_d^2)\lambda_{1:T}^* + \sum\limits_{t=1}^{T}\frac{(M_t +\lambda_t^*M_d)^2}{\mu_{1:t}+\lambda_{1:t}^*}\right),
    \end{aligned}
\end{equation*}
provided the $\lambda_t$ are chosen as solutions to
$$
(A^2+2M_d^2)\lambda_{t}=\frac{2M_t^2}{\mu_{1:t}+\lambda_{1:t-1}+\lambda_t}.
$$

It is easy to verify that
$$
\lambda_t=\frac{1}{2}\left(\sqrt{(\mu_{1:t}+\lambda_{1:t-1})^2+8M_t^2/(A^2+2M_d^2)}-(\mu_{1:t}+\lambda_{1:t-1})\right)
$$
is the non-negative root of the above quadratic equation.
\end{proof}

\section*{Appendix C. The proof of Theorem \ref{Thm5}.}

\begin{proof}
By assumption on the functions $f_t$ and $d$ for every productive step we have
\begin{equation*}
    \begin{aligned}
        & \eta_t((f_t(x_t)+\lambda_t d(x_t))-(f_t(x^*)+\lambda_td(x^*)))
        \\&
        \leq \eta_t(\langle \nabla f_t(x_t)+\lambda_t\nabla d(x_t), x_t -x^*\rangle - (\mu_t+\lambda_t) V(x^*,x_t))
        \\&
        \leq \eta_t^2(M_t+\lambda_tM_d)^2+V(x^*,x_t)-V(x^*,x_{t+1})-\eta_t(\mu_t+\lambda_t) V(x^*,x_t).
    \end{aligned}
\end{equation*}
Hence, after dividing both sides of the above inequality by $\eta_t$ we get
\begin{equation*}
    \begin{aligned}
        & \quad (f_t(x_t)+\lambda_t d(x_t))-(f_t(x^*)+\lambda_td(x^*))
        \\&
        \leq\eta_t(M_t+\lambda_tM_d)^2+\frac{1}{\eta_t}\left(V(x^*,x_t)-V(x^*,x_{t+1})\right)-(\mu_t+\lambda_t)V(x^*,x_t)
        \\&
        = \frac{(M_t+\lambda_tM_d)^2}{\mu_{1:t}+\lambda_{1:t}} + (\mu_{1:t}+\lambda_{1:t})V(x^*,x_t)-(\mu_t+\lambda_t) V(x^*,x_t) -
        \\& \;\;\;\; -(\mu_{1:t}+\lambda_{1:t})V(x^*,x_{t+1})
        \\&
        = \frac{(M_t + \lambda_t M_d)^2}{\mu_{1:t} + \lambda_{1:t}} + (\mu_{1:t-1} + \lambda_{1:t-1}) V(x^*, x_t) -(\mu_{1:t} + \lambda_{1:t}) V(x^*, x_{t+1}).
    \end{aligned}
\end{equation*}

Similarly, taking into account the $M_g$-relative Lipschitz-continuity of $g$ and the $M_d$-relative Lipschitz-continuity of $d$  for every non-productive step we have $g(x_t) > \varepsilon$, and
\begin{equation*}
    \begin{aligned}
       \eta_t\varepsilon & < \eta_t((g(x_t)+\lambda_t d(x_t))-(g(x^*)+\lambda_td(x^*)))
        \\&
        \leq\eta_t \left(\langle\nabla g(x_t)+\lambda_t\nabla d(x_t), x_t -x^*\rangle - (\mu_t+\lambda_t) V(x^*,x_t) \right)
        \\&
        \leq \eta_t^2 (M_g + \lambda_t M_d)^2 + V(x^*,x_t) - V(x^*, x_{t+1}) - \eta_t (\mu_t + \lambda_t) V(x^*, x_t).
    \end{aligned}
\end{equation*}

Dividing both sides of the last inequality by $\eta_t$, we get:
\begin{equation*}
    \begin{aligned}
        \varepsilon & < (g(x_t)+\lambda_t d(x_t))-(g(x^*)+\lambda_td(x^*)
        \\&
        \leq \eta_t (M_g+\lambda_tM_d)^2 +\frac{1}{\eta_t}\left(V(x^*,x_t)-V(x^*,x_{t+1})\right)-(\mu_t+\lambda_t)V(x^*,x_t)
        \\&
        = \frac{(M_g+\lambda_tM_d)^2}{\mu_{1:t}+\lambda_{1:t}} + (\mu_{1:t}+\lambda_{1:t})V(x^*,x_t)-(\mu_t+\lambda_t)V(x^*,x_t) -
        \\& \;\;\;\;- (\mu_{1:t}+\lambda_{1:t}) V(x^*,x_{t+1})
        \\&
        =\frac{(M_g + \lambda_tM_d)^2}{\mu_{1:t} + \lambda_{1:t}} + (\mu_{1:t-1} + \lambda_{1:t-1})V(x^*,x_t) -(\mu_{1:t} + \lambda_{1:t}) V(x^*, x_{t+1}).
    \end{aligned}
\end{equation*}

Summing up the inequalities for productive and non-productive steps, and let $M = \max\{M_t,M_g\}$, then
\begin{equation*}
    \begin{aligned}
        & \sum_{t \in I} ((f_t(x_t)+\lambda_t d(x_t))-(f_t(x^*)+\lambda_td(x^*))) + \sum_{t \in J} ((g(x_t)+\lambda_t d(x_t))-(g(x^*)+\lambda_td(x^*)))
        \\&
        \leq  \sum_{t=1}^{ T+T_J} \left( \frac{(M+\lambda_tM_d)^2}{\mu_{1:t}+\lambda_{1:t}} + (\mu_{1:t-1}+\lambda_{1:t-1})V(x^*,x_t)-(\mu_{1:t}+\lambda_{1:t})V(x^*,x_{t+1})\right)
        \\&
        \leq \sum_{t=1}^{ T+T_J}\frac{(M+\lambda_tM_d)^2}{\mu_{1:t}+\lambda_{1:t}} - (\mu_{1:T+T_J}+\lambda_{1:T+T_J})V(x^*,x_{T+T_J})
        \\&
        \leq \sum_{t=1}^{ T+T_J}\frac{(M+\lambda_tM_d)^2}{\mu_{1:t}+\lambda_{1:t}}.
    \end{aligned}
\end{equation*}

Bounding $d(x^*)\le A^2$ and using the fact, that for non-productive steps
$$ g(x_t) - g(x^*) \ge g(x_t) > \varepsilon, $$
we get an estimate for the sum of the objective functionals:
\begin{equation*}
    \begin{aligned}
        \sum_{t=1}^T (f_t(x_t) - f_t(x^*)) & = \sum\limits_{t =1}^T f_t(x_t) - \min\limits_{x\in Q}\sum\limits_{t=1}^T f_t(x)
        \\&
        \leq \sum_{t=1}^{ T+T_J}\frac{(M+\lambda_tM_d)^2}{\mu_{1:t} +\lambda_{1:t}} - \sum_{t \in J} (g(x_t) - g(x^*))
        \\&
        \;\;\;\; +\sum\limits_{t =1}^{T+T_J}\lambda_td(x^*)-\sum\limits_{t =1}^{T+T_J}\lambda_td(x_t)
        \\&
        \leq\sum_{t=1}^{ T+T_J}\frac{(M+\lambda_tM_d)^2}{\mu_{1:t}+\lambda_{1:t}}+\lambda_{1:T+T_J}A^2 - \varepsilon T_J.
    \end{aligned}
\end{equation*}

Thus, we get
$$0\leq Regret_T\leq\sum_{t=1}^{ T+T_J}\frac{(M+\lambda_tM_d)^2}{\mu_{1:t}+\lambda_{1:t}}+\lambda_{1:T+T_J}A^2 - \varepsilon T_J.$$
Using inequality \eqref{X3}, we have
\begin{equation}\label{X33}
\begin{aligned}
&\quad \lambda_{1:T+T_J}A^2 + \sum\limits_{t=1}^{T+T_J} \frac{(M +\lambda_tM_d)^2}{\mu_{1:t} + \lambda_{1:t}} - \varepsilon T_J
\\&
\le \lambda_{1:T+T_J}A^2+\sum\limits_{t=1}^{T+T_J}\left(\frac{2M^2}{\mu_{1:t}+\lambda_{1:t}}+\frac{2\lambda_t^2M_d^2}{\mu_{1:t}+\lambda_{1:t-1}+\lambda_t}\right)-\varepsilon T_J
\\&
\le (A^2+2M_d^2)\lambda_{1:T+T_J}+2\sum\limits_{t=1}^{T+T_J}\frac{M^2}{\mu_{1:t}+\lambda_{1:t}}-\varepsilon T_J.
\end{aligned}
\end{equation}
By \eqref{X33} and Lemma \ref{l_1}
\begin{equation*}
    \begin{aligned}
    Regret_T & \le(A^2+2M_d^2) \lambda_{1:T+T_J} + 2 \sum\limits_{t=1}^{T+T_J} \frac{M^2}{\mu_{1:t}+\lambda_{1:t}}-\varepsilon T_J
    \\&
    \le\inf\limits_{\lambda_1^*,\dots,\lambda_{T+T_J}^*}\left(2(A^2+2M_d^2)\lambda_{1:T+T_J}^*+4\sum\limits_{t=1}^{T+T_J}\frac{M^2}{\mu_{1:t}+\lambda_{1:t}^*}\right)-\varepsilon T_J
    \\&
    \le 2 \inf\limits_{\lambda_1^*, \dots,\lambda_{T+T_J}^*} \left((A^2+2M_d^2) \lambda_{1:T+T_J}^* + \sum\limits_{t=1}^{T + T_J} \frac{(M + \lambda_t^*M_d)^2}{\mu_{1:t}+\lambda_{1:t}^*}\right)-\varepsilon T_J.
    \end{aligned}
\end{equation*}
provided the $\lambda_t$ are chosen as solutions to
$$
(A^2+2M_d^2)\lambda_{t}=\frac{2M^2}{\mu_{1:t}+\lambda_{1:t-1}+\lambda_t}.
$$

It is easy to verify that
$$
\lambda_t=\frac{1}{2}\left(\sqrt{(\mu_{1:t}+\lambda_{1:t-1})^2+8M^2/(A^2+2M_d^2)}-(\mu_{1:t}+\lambda_{1:t-1})\right)
$$
is the non-negative root of the above quadratic equation.
\end{proof}



\section*{Appendix D. The proof of Corollary \ref{cor_4}.}

\begin{proof}
1. Indeed, if $\lambda_t=0 \; \forall 1\leq t\leq T+T_J$, then the claimed statement immediately follows from Corollary \ref{cor_1}.

2. Indeed, if $\lambda_1=\sqrt{T+T_J}$ and $\lambda_t=0$ for $1<t\leq T+T_J,$ then
\begin{equation*}
    \begin{aligned}
        0 & \leq Regret_T \le \lambda_{1:T+T_J}A^2 + \sum\limits_{t=1}^{T+T_J} \frac{( M + \lambda_tM_d)^2}{\mu_{1:t}+\lambda_{1:t}}-\varepsilon T_J
        \\&
        \le(A^2+2M_d^2)\lambda_{1:T+T_J}+2\sum\limits_{t=1}^{T+T_J}\frac{M^2}{\mu_{1:t}+\lambda_{1:t}}-\varepsilon T_J
        \\&
        \leq(A^2+2M_d^2)\sqrt{T+T_J} +2\sum\limits_{t=1}^{T+T_J}\frac{M^2}{\sqrt{T+T_J}}-\varepsilon T_J
        \\&
        =\left(A^2+2(M_d^2+M^2)\right)\sqrt{T+T_J}-\varepsilon T_J,
    \end{aligned}
\end{equation*}
hence $\varepsilon T_J\leq\left(A^2+2(M_d^2+M^2)\right)\sqrt{T+T_J}.$ Let  $\varepsilon=\dfrac{A^2+2(M_d^2+M^2)}{\sqrt{T}}$.  Then we get
$$
\dfrac{T_J}{T}\leq \sqrt{\dfrac{T+T_J}{T}}=\sqrt{1+\dfrac{T_J}{T}}.
$$
Since the linear function grows faster than the square root function, it is obviously, that with a sufficiently large $T_J$, the above inequality does not hold, and then $\dfrac{T_J}{T}$ is bounded. Thus we proved that $\exists \  C>0:\  T_J \leq C \cdot T$.
So, we have $$ Regret_T\leq \left(A^2+2(M_d^2+M^2)\right)\sqrt{(C+1)T}=O(\sqrt{T}).$$

3.  \al{Let us} assume
$\lambda_1 = (T+T_J)^{\alpha}, \lambda_t=0, \; \forall 1\leq t \leq T+T_J$. Note that
$$
    \mu_{1:t}:=\sum\limits_{s=1}^{t}\mu_s\geq \int\limits_{0}^{t-1}(x+1)^{-\alpha}dx=(1-\alpha)^{-1}\left(t^{1-\alpha}-1\right). $$
Hence
\begin{equation*}
    \begin{aligned}
      0 & \leq Regret_T\leq \lambda_{1:T+T_J}A^2+\sum\limits_{t=1}^{T+T_J}\frac{(M+\lambda_tM_d)^2}{\mu_{1:t}+\lambda_{1:t}}-\varepsilon T_J
      \\&
      \le(A^2+2M_d^2) \lambda_{1:T+T_J} + 2 \sum\limits_{t=1}^{T+T_J} \frac{M^2}{\mu_{1:t} + \lambda_{1:t}} - \varepsilon T_J
      \\&
      \leq (A^2+2M_d^2)(T+T_J)^{\alpha}+2 M^2 (1-\alpha) \sum_{t=1}^{T+T_J}\frac{1}{\left(t^{1-\alpha}-1\right)} - \varepsilon T_J
      \\&
      \leq (A^2 + 2 M_d^2)(T+T_J)^{\alpha}+4 M^2\frac{1}{\alpha}(T + T_J)^\alpha + O(1) -\varepsilon T_J.
    \end{aligned}
\end{equation*}

Then we have $\varepsilon T_J\leq(A^2+2M_d^2+4M^2\frac{1}{\alpha})(T+T_J)^\alpha.$ Let $\varepsilon=(A^2+2M_d^2+4M^2\frac{1}{\alpha})\dfrac{T^\alpha}{T},$ then
$$
\dfrac{T^\alpha}{T}T_J\leq (T+T_J)^\alpha,
$$
and
$$
\dfrac{T_J}{T}\leq \left(\dfrac{T+T_J}{T}\right)^\alpha=\left(1+\dfrac{T_J}{T}\right)^\alpha.
$$
It is obviously, that with a sufficiently large $T_J$, the above inequality does not hold, and then $\exists \  C>0:\  T_J \leq C \cdot T$.
Thus, we have $$ Regret_T=O((T+T_J)^\alpha)=O(((C+1)T)^\alpha)=O(T^\alpha).$$
\end{proof}

\end{document}